\documentclass[12pt]{amsart}
\usepackage{amssymb}
\usepackage[all]{xy}

\newtheorem{theorem}{Theorem}[section]
\newtheorem{definition}[theorem]{Definition}
\newtheorem{proposition}[theorem]{Proposition}
\newtheorem{conjecture}[theorem]{Conjecture}

\begin{document}

\title{Thoma type results for discrete quantum groups}

\author{Teodor Banica}
\address{T.B.: Department of Mathematics, University of Cergy-Pontoise, F-95000 Cergy-Pontoise, France. {\tt teo.banica@gmail.com}}

\author{Alexandru Chirvasitu}
\address{A.C.: Department of Mathematics, Box 354350, University of Washington, Seattle, WA 98195, USA. {\tt chirva@math.washington.edu}}

\begin{abstract}
Thoma's theorem states that a group algebra $C^*(\Gamma)$ is of type I if and only if $\Gamma$ is virtually abelian. We discuss here some similar questions for the quantum groups, our main result stating that, under suitable virtually abelianity conditions on a discrete quantum group $\Gamma$, we have a stationary model of type $\pi:C^*(\Gamma)\to M_F(C(L))$, with $F$ being a finite quantum group, and with $L$ being a compact group. We discuss then some refinements of these results in the quantum permutation group case, $\widehat{\Gamma}\subset S_N^+$, by restricting the attention to the matrix models which are quasi-flat, in the sense that the images of the standard coordinates, known to be projections, have rank $\leq1$.
\end{abstract}

\subjclass[2010]{46L54 (60B15)}
\keywords{Quantum group, Matrix model, Stationarity}

\maketitle

\section*{Introduction}

Thoma's theorem \cite{tho} states that a group algebra $C^*(\Gamma)$ is of type I, in the sense that we have a $C^*$-algebra embedding $\pi:C^*(\Gamma)\subset M_K(C(X))$, with $K<\infty$, and with $X$ being a compact space, if and only if $\Gamma$ is virtually abelian, in the sense that we have a normal abelian subgroup $\Lambda\triangleleft\Gamma$ such that the quotient group $\Phi=\Gamma/\Lambda$ is finite.

This statement is of interest in connection with von Neumann's work on the operator algebras, and more specifically with his reduction theory for such algebras, and with his classification of factors, which are the building blocks of the theory, into three types: I, II, III. Indeed, from von Neumann's point of view, the ``simplest'' operator algebras, and in particular the simplest group algebras, are those of type I. See \cite{mvo}, \cite{von}.

In this paper we discuss a number of refinements, and quantum group extensions, of Thoma's theorem. We use Woronowicz's formalism \cite{wo1}, \cite{wo2}, \cite{wo3}, and our goal is that of finding necessary and sufficient conditions on a compact quantum group $G$ which ensure the existence of a random matrix model of the following type:
$$\pi:C(G)\to M_F(C(L))$$

Here $F$ is a finite quantum group, $L$ is a compact group, and $\pi$ is subject to the ``stationarity'' condition $\int_G=(tr\otimes\int_L)\pi$, where $tr$ is the normalized matrix trace. We restrict the attention to such models because in the group dual case, $G=\widehat{\Gamma}$, this is what comes out, with $F=\widehat{\Phi}$ and $L=\widehat{\Lambda}$, from the virtual abelianity condition on $\Gamma$.

Mathematically speaking, our work is motivated by the above-mentioned von Neumann philosophy, applied to the quantum group setting, and also by the abstract interactions between quantum groups and random matrix theory. Physically speaking, the stationarity condition, coming from the idempotent state work in \cite{fsk}, \cite{wa3}, is expected to correspond to a natural ``stationarity'' condition on the partition function of the associated 2D model, so the study of the stationary models is of particular interest too. We refer to \cite{ban}, \cite{bfr}, \cite{bne} for some recent work on the subject, that the present paper continuates.

Generally speaking, the full quantum group extension of Thoma's theorem appears as a difficult question, and even formulating a conjectural statement is not an easy task, due to the numerous obstructions which appear. We will present here, however, a number of fairly general results on the subject, basically extending everything that is known. We intend to come back to these questions, with finer results, in some future papers.

The paper is organized as follows: in 1-3 we discuss Thoma's theorem and its quantum extensions, and in 4-6 we present a number of more specialized results on the subject.

\medskip

\noindent {\bf Acknowledgements.} A.C. is grateful for partial support from the NSF through grant DMS-1565226.

\section{Thoma's theorem, revisited}

We are interested in the matrix models for the discrete group algebras, which are stationary in the following sense:

\begin{definition}
A matrix model $\pi:C^*(\Gamma)\to M_K(C(X))$ is called stationary when $X$ is a compact probability space, and we have
$$\int_{\widehat{\Gamma}}=\left(tr\otimes\int_X\right)\pi$$
via the identification $M_K(C(X))\simeq M_K(\mathbb C)\otimes C(X)$, where $\int_{\widehat{\Gamma}}:C^*(\Gamma)\to\mathbb C$ is given by $g\to\delta_{g,1}$, and where $tr:M_K(\mathbb C)\to\mathbb C$ is the normalized matrix trace.
\end{definition}

Here $K\in\mathbb N$ is a positive integer, but for simplicity of presentation we agree to use as well finite sets $K$, with the matrix convention $M_K=M_{|K|}$. Later on, we will use as well finite quantum groups $K$, once again with the matrix convention $M_K=M_{|K|}$, where the cardinality of $K$ is given by definition by the following formula:
$$|K|=\dim_\mathbb CC(K)$$

Observe that a stationary model is faithful. Indeed, the stationarity condition gives a factorization $\pi:C^*(\Gamma)\to C^*_{red}(\Gamma)\subset M_K(C(X))$, and since the reduced algebra $C^*_{red}(\Gamma)$ follows to be of type I, $\Gamma$ must be amenable, and so $\pi$ is faithful. See \cite{bfr}.

The main result for the discrete groups, that we would like first to refine, and then to generalize to the quantum group setting, is Thoma's theorem \cite{tho}. This theorem states that a group algebra $C^*(\Gamma)$ is of type I if and only if $\Gamma$ is virtually abelian.

We have the following more detailed version of this theorem, by using the notion of stationarity introduced above:

\begin{theorem}[Thoma]
For a discrete group $\Gamma$, the following are equivalent:
\begin{enumerate}
\item $C^*(\Gamma)$ is of type I, in the sense that we have an embedding $\pi:C^*(\Gamma)\subset M_K(C(X))$, with $X$ being a compact space.

\item $C^*(\Gamma)$ has a stationary model of type $\pi:C^*(\Gamma)\to M_\Phi(C(L))$, with $\Phi$ being a finite group, and $L$ being a compact abelian group.

\item $\Gamma$ is virtually abelian, in the sense that we have an abelian subgroup $\Lambda\triangleleft\Gamma$ such that the quotient group $\Phi=\Gamma/\Lambda$ is finite. 

\item $\Gamma$ has an abelian subgroup $\Lambda\subset\Gamma$ whose index $K=[\Gamma:\Lambda]$ is finite.
\end{enumerate}
\end{theorem}

\begin{proof}
There are several proofs for this fact, the idea being as follows:

$(1)\implies(4)$ This is the non-trivial implication, see \cite{kan}, \cite{smi}, \cite{tho}.

$(4)\implies(3)$ We choose coset representatives $g_i\in\Gamma$, and we set:
$$\Lambda'=\bigcap_ig_i\Gamma g_i^{-1}$$

Then $\Lambda'\subset\Lambda$ has finite index, and we have $\Lambda'\triangleleft\Gamma$, as desired.

$(3)\implies(2)$ This follows by using the theory of induced representations. We can define a model $\pi:C^*(\Gamma)\to M_\Phi(C(\widehat{\Lambda}))$ by setting:
$$\pi(g)(\chi)=Ind_\Lambda^\Gamma(\chi)(g)$$

Indeed, any character $\chi\in\widehat{\Lambda}$ is a 1-dimensional representation of $\Lambda$, and we can therefore consider the induced representation $Ind_\Lambda^\Gamma(\chi)$ of the group $\Gamma$. This representation is $|\Phi|$-dimensional, and so maps the group elements $g\in\Gamma$ into order $|\Phi|$ matrices $Ind_\Lambda^\Gamma(\chi)(g)$. Thus $\pi$ is well-defined, and the fact that it is a representation is clear as well.

In order to check now the stationarity property, we use the following well-known character formula, due to Frobenius:
$$Tr\left(Ind_\Lambda^\Gamma(\chi)(g)\right)=\sum_{x\in\Phi}\delta_{x^{-1}gx\in\Lambda}\chi(x^{-1}gx)$$

By integrating with respect to $\chi\in\widehat{\Lambda}$, we deduce from this that we have:
$$ \left(Tr\otimes \int_{\widehat{\Lambda}}\right)\pi(g)=\sum_{x\in\Phi}\delta_{x^{-1}gx\in \Lambda}\int_{\widehat{\Lambda}}\chi(x^{-1}gx)d\chi=\sum_{x\in\Phi}\delta_{x^{-1}gx\in \Lambda}\delta_{g,1}=|\Phi|\cdot\delta_{g,1}$$

Now by dividing by $|\Phi|$ we conclude that the model is stationary, as claimed.

$(2)\implies(1)$ This is the trivial implication, with the faithfulness of $\pi$ following from the abstract functional analysis arguments explained after Definition 1.1.
\end{proof}

Summarizing, we have a good understanding of the ``virtually abelian implies type I'' part of Thoma's theorem, corresponding to the proof of $(3)\implies(2)$ above. In what follows we will be mainly interested in refining and extending this result. 

\section{Quantum groups}

We present in what follows a number of quantum group extensions of the above results. Generally speaking, we are in need here of some induced representation machinery for the discrete quantum groups. Some results in this direction are already available from \cite{vvo}, but for the purposes of the present paper, where we will make a heavy use of exact sequences in the $*$-algebraic setting, we will rather use \cite{ade}, \cite{mon} as main ingredients.

We use the formalism of compact and discrete quantum groups due to Woronowicz \cite{wo3}, under the supplementary assumption $S^2=id$. The axioms here are as follows:

\begin{definition}
Assume that we have a unital $C^*$-algebra $A$, together with a morphism  of $C^*$-algebras $\Delta:A\to A\otimes A$, such that the following conditions are satisfied:
\begin{enumerate}
\item $\Delta$ is coassociative: $(id\otimes\Delta)\Delta=(\Delta\otimes id)\Delta$.

\item $A$ is left and right simplifiable, with respect to $\Delta$.

\item The associated antipode map $S$ satisfies $S^2=id$.
\end{enumerate}
We write then $A=C(G)=C^*(\Gamma)$, and call $G,\Gamma$ compact and discrete quantum groups.
\end{definition}

Observe that if $G$ is a compact group then $A=C(G)$ satisfies the axioms, and that if $\Gamma$ is a discrete group then $A=C^*(\Gamma)$ satisfies the axioms. Together with a number of other observations and results, including the Pontrjagin duality formulae $G=\widehat{\Gamma},\Gamma=\widehat{G}$, valid in the abelian case, this justifies the above axioms, and terminology. See \cite{wo3}.

For the purposes of the present paper, the above formalism is exactly what we need, and we will formulate everything in terms of compact quantum groups.

Let us begin our study with an extension of Definition 1.1, as follows:

\begin{definition}
Given a compact quantum group $G$, a matrix model $\pi:C(G)\to M_K(C(X))$ is called stationary when $X$ is a compact probability space, and
$$\int_G=\left(tr\otimes\int_X\right)\pi$$
via the standard identification $M_K(C(X))\simeq M_K(\mathbb C)\otimes C(X)$.
\end{definition}

Here $\int_G$ is the Haar integration over $G$, constructed by Woronowicz in \cite{wo3}.

Observe that in the group dual case, $G=\widehat{\Gamma}$ with $\Gamma$ classical, we recover Definition 1.1 above. As in the group dual case, the stationarity implies the faithfulness, by using some standard amenability theory from \cite{bmt}. For a discussion of the stationarity property, and of related matrix modelling questions, we refer to \cite{ban}, \cite{bb2}, \cite{bfr}, \cite{chi}.

In order to extend Thoma's theorem, let us begin with some algebraic preliminaries. It is convenient to restrict our attention to the dense Peter-Weyl type Hopf $*$-algebras $\mathcal{R}(G)\subset C(G)$ associated to our compact quantum groups $G$, as in \cite{wo3}. 

Following \cite{ade}, we have the following key definition: 

\begin{definition}
An exact sequence of compact quantum groups $1\to F\to G\to L\to1$ corresponds by definition to a sequence of Hopf $*$-algebras
$$\mathbb C\to \mathcal{R}(L)\stackrel{\iota}{\to} \mathcal{R}(G)\stackrel{\psi}{\to} \mathcal{R}(F)\to\mathbb C$$
which is exact, in the following sense:
\begin{enumerate}
\item $\iota$ is injective, $\psi$ is surjective.

\item $\psi\circ\iota$ equals the composition of the counit of $\mathcal{R}(L)$ and the unit of $\mathcal{R}(F)$.

\item $\mathrm{ker}(\psi)=\mathcal{R}(G)\mathcal{R}(L)^+$, where the $+$ superscript denotes the kernel of the counit.

\item $Im(\iota)=\{x\in\mathcal{R}(G)|(\psi\otimes 1)\Delta(x)=1\otimes x\}$.
\end{enumerate}
\end{definition}

We will need as well to consider the more general situation where we have a quotient $G\to L$, without necessarily having an exact sequence as above. Such a quotient $G\to L$ produces at the dual level an inclusion $\widehat{L}\subset\widehat{G}$, and with some inspiration from the condition (3) above, we can define the ``index'' of this inclusion as follows:
$$[\widehat{G}:\widehat{L}]=\dim_\mathbb C\left[\mathcal{R}(G)\big/\mathcal{R}(G)\mathcal{R}(L)^+\right]$$

To be more precise, the quotient operation on the right is a coalgebra quotient. Observe that in the case where we have an exact sequence, as in Definition 2.3, the index is simply the quantity $|F|=\dim_\mathbb C\mathcal R(F)$. In general, the above quantity $[\widehat{G}:\widehat{L}]$, that we will also call ``co-index'' of the quotient $G\to L$, corresponds to what we would like to expect from an index for discrete quantum groups. For more on these topics, we refer to \cite{ade}.

Following now \cite{mon}, and more precisely Definition 7.2.1 there, we have:

\begin{definition}
An exact sequence $1\to F\to G\to L\to 1$ is called ``cleft'' if the extension $\mathcal{R}(L)\subset \mathcal{R}(G)$ is cleft with respect to the right coaction of $\mathcal{R}(F)$ on $\mathcal{R}(G)$, in the sense that there exists a convolution-invertible right $\mathcal{R}(F)$-comodule map $\mathcal{R}(F)\to \mathcal{R}(G)$.
\end{definition}

The idea with the cleftness condition is that this is a natural and technically useful generalization of the ``trivial'' situation, where the short exact sequence is split. For more theory regarding the short exact sequences, we refer to \cite{ade}, \cite{chk}, \cite{cdp}, \cite{kta}, \cite{mon}.

With these notions in hand, we can explore the various notions of virtual abelianity which can be imposed on a discrete quantum group $\Gamma=\widehat{G}$. We first have the following result, which is a trivial consequence of the above definitions:

\begin{proposition}
Consider the following notions of virtual coabelianity, concerning a compact quantum group $G$:
\begin{enumerate}
\item We have a finite co-index quotient $G\to L$, with $L$ classical.

\item We have an exact sequence $1\!\to\!F\!\to\!G\!\to\!L'\!\to\!1$ with $F$ finite, and $L\!\to\!L'$ classical.

\item We have an exact sequence $1\to F\to G\to L\to1$, with $F$ finite, $L$ classical.

\item We have a cleft sequence $1\to F\to G\to L\to1$, with $F$ finite, $L$ classical.
\end{enumerate}
We have then $(4)\implies(3)\implies(2)\implies(1)$. In addition, in the group dual case, the implications $(1)\implies(2)$ and $(3)\implies(4)$ hold as well. 
\end{proposition}

\begin{proof}
The implications $(4)\implies(3)\implies(2)\implies(1)$ are trivial. As for the group dual case, $(1)\implies(2)$ is essentially the implication $(4)\implies(3)$ from Theorem 1.2 above, and $(3)\implies(4)$ follows from the fact that the comodule morphism $\mathcal{R}(F)\to \mathcal{R}(G)$ that we need can be chosen to be a section of the quotient map $\widehat{G}\to \widehat{F}$.
\end{proof}

In connection now with Thoma's theorem, or rather with its constructive implication, ``virtually abelian implies type I'', we first have the following result:

\begin{proposition}
Assuming that $G$ has a classical finite co-index quotient $G\to L$ fitting into an exact sequence as in Definition 2.3, we have a faithful representation
$$\pi:\mathcal{R}(G)\to End_{\mathcal{R}(L)}\mathcal{R}(G)$$
which in addition commutes with the integration functionals, in the sense that
$$\int_G=\int_L\circ\;tr\circ\pi$$
where $tr:End_{\mathcal{R}(L)}\mathcal{R}(G)\to \mathcal{R}(L)$ is the canonical splitting of $\mathcal{R}(L)\subset End_{\mathcal{R}(L)}\mathcal{R}(G)$.
\end{proposition}

\begin{proof}
By Definition 2.3 we have an exact sequence, as follows:
$$\mathbb C\to \mathcal{R}(L)\stackrel{\iota}{\to} \mathcal{R}(G)\stackrel{\psi}{\to} \mathcal{R}(F)\to\mathbb C$$

We now regard $\mathcal{R}(G)$ as a right $\mathcal{R}(F)$-comodule algebra via the following coaction:
$$(1\otimes\psi)\circ \Delta: \mathcal{R}(G)\to \mathcal{R}(G)\otimes \mathcal{R}(F)$$

Since $\mathcal{R}(F)$ is cosemisimple, all of its comodules are left and right coflat. By \cite{tak} the functor $-\otimes_{\mathcal{R}(L)}\mathcal{R}(G)$ is an equivalence between the category of right $\mathcal{R}(L)$-modules and the category $\mathcal{M}_{\mathcal{R}(G)}^{\mathcal{R}(F)}$ of $\mathcal{R}(G)$-modules $M$ equipped with a right $\mathcal{R}(F)$-comodule structure for which the module structure map $M\otimes\mathcal{R}(G)\to M$ is one of $\mathcal{R}(F)$-comodules.

By Theorem I of \cite{sch}, the extension $\mathcal{R}(L)\subset \mathcal{R}(G)$ is Galois over $\mathcal{R}(F)$ in the sense of Definition 8.1.1 in \cite{mon}. Now Theorem 8.3.3 in \cite{mon} implies that $\mathcal R(G)$ is projective finitely generated as a $\mathcal R(L)$-module, and so we have an isomorphism, as follows:
$$\Xi:\mathcal R(G)\otimes_{\mathcal R(L)} Hom_{\mathcal R(L)}(\mathcal R(G),\mathcal R(L))\simeq End_{\mathcal R(L)}(\mathcal R(G))$$

With this picture in mind, we define $\pi$ to be the map identifying $\mathcal R(G)$ with the space of $\mathcal R(L)$-module maps $\mathcal R(G)\to\mathcal R(G)$ given by left multiplication by elements in $\mathcal R(G)$.

In order to prove the last assertion, we need an explicit description of the normalized trace. To this end we paraphrase the discussion in the proof of Theorem 8.3.1 in \cite{mon}, as follows. Consider the following diagram, well-known to commute:
$$\xymatrix@R=35pt@C=25pt
{\mathcal R(G)\otimes\mathcal R(G)\otimes\mathcal R(G)\ar[rr]^{m\otimes id}&&\mathcal R(G)\otimes\mathcal R(G)\ar[d]^{id\otimes p}\\
\mathcal R(G)\otimes\mathcal R(G)\ar[u]^{id\otimes\Delta}\ar[r]&\mathcal R(G)\otimes_{\mathcal R(L)}\mathcal R(G)\ar[r]_\simeq&\mathcal R(G)\otimes \mathcal R(F)}$$  

Let $\int_{\widehat{F}}\in \mathcal{R}(F)$ be the unique integral with $\varepsilon\left(\int_{\widehat{F}}\right)=1$, i.e. the Haar state of the compact quantum group $\widehat{F}$. With the convention that we sum over repeated indices, consider the unique element $a_i\otimes b_i\in\mathcal R(G)\otimes_{\mathcal R(L)}\mathcal R(G)$ which gets mapped to $1\otimes\int_{\widehat{F}}\in\mathcal R(G)\otimes\mathcal R(F)$ by the two equal bottom-left-to-bottom-right maps in the above diagram.

Unwinding the proof of Theorem 8.3.1 in \cite{mon}, the element 
corresponding to the identity through the isomorphism $\Xi$ constructed above is $a_i\otimes\phi_i$, where the elements $\phi_i\in Hom(\mathcal{R}(G)\to \mathcal{R}(L))$ are defined by the following formula: 
$$\phi(y) = b_{i1}y_{i1}\int_F\psi(b_{i2}y_{i2})$$

In short, we have identified the left multiplication by $x\in\mathcal R(G)$ with: 
$$xa_i\otimes\phi_i\in\mathcal R(G)\otimes_{\mathcal R(L)}Hom(\mathcal R(G),\mathcal R(L))$$

We conclude from this that we have the following equality:
$$tr(\pi(x))=\phi_i(xa_i)=b_{i1}x_{i1}a_{i1}\int_F\psi(b_{i2}x_{i2}a_{i2})$$

Applying now $\int_L$ and setting $y=b_ixa_i$, we obtain:
\begin{eqnarray*}
\int_L tr(\pi(x))
&=&\int_Ly_1\int_F\psi(y_2)=\int_Gy=\int_Gb_ixa_i=\int_Gxa_ib_i\\
&=&\left(\int_G\otimes\varepsilon\right)\left(xa_ib_{i1}\otimes b_{i2}\right)=\left(\int_G\otimes\varepsilon\right)\left[x\left(1\otimes\int_{\widehat{F}}\right)\right]=\int_Gx
\end{eqnarray*}

We conclude that we have $\int_G=\int_L\circ tr\circ \pi$, as desired.
\end{proof}

In order to examine now how the $*$-structures enter the discussion, consider the projection $E:\mathcal{R}(G)\to \mathcal{R}(L)$ obtained by annihilating the summands of the domain that do not appear in the codomain under the Peter-Weyl decomposition:
$$\mathcal{R}(L)=\bigoplus_{V\in \mathrm{Irr}(L)}V^*\otimes V\quad,\quad \mathcal{R}(G)=\bigoplus_{W\in \mathrm{Irr}(G)} W^*\otimes W$$

This projection $E$ extends to a projection of $C(G)\to C(L)$, and this in turn realizes the former as a Hilbert module over the latter, with the following pairing:
$$\langle x|y\rangle = E(x^*y)\in C(L),\ \forall x,y\in C(G)$$ 

Summarizing, we have a $*$-structure on the adjointable operators on $C(G)$ respecting the right $C(L)$-module structure that restricts to the codomain $End_{\mathcal{R}(L)}\mathcal{R}(G)$ of the morphism $\pi$ in Proposition 2.6. Since $\mathcal{R}(G)$ is realized as an algebra of such operators by left multiplication, the $\pi$ is easily seen to respect the $*$-structures.

With all of this in place, we can now lift the above results to the $C^*$-setting:

\begin{theorem}
Under the hypotheses of Proposition 2.6 the discrete quantum dual $\widehat{G}$ is amenable, and we have a stationary embedding of $C(G)$ into the $C^*$-algebra $End_{C(L)}C(G)$ of adjointable operators of the Hilbert $C(L)$-module $C(G)$. 
\end{theorem}

\begin{proof}
The amenability follows as in Proposition 4.4 in \cite{bfr}, and the result itself follows by lifting the map $\pi$ from Proposition 2.6 to the corresponding $C^*$-algebra $C(G)$.
\end{proof}

Converting the above result into a true stationarity statement is a quite subtle task, which raises a number of interesting algebraic questions. We will discuss here some of these questions, and we intend to perform a more systematic study in a future paper.

\section{Cleft extensions}

In the context of Theorem 2.7, the simplest situation is that where $C(G)$ is free as a module over $C(L)$. Indeed, assuming that it is so, the target algebra becomes a usual matrix algebra, and Theorem 2.7 itself becomes a usual stationarity statement.

Now in order to have this freeness property, we are led to the cleft sequences, and to the strongest notion of virtual abelianity, from Proposition 2.5 (4). To be more precise, the result which can be extracted in this way from Theorem 2.7 is as follows:

\begin{theorem}
Assuming that we have a cleft sequence $1\to F\to G\to L\to1$, with $F$ being finite, and with $L$ being classical, we have a matrix model
$$\pi:C(G)\to M_F(C(L))$$
which is stationary, in the sense of Definition 2.2 above.
\end{theorem}

\begin{proof}
According to Definition 2.4 above, our assumption is that we have a cleft sequence as follows, with $F$ finite and $L$ classical:
$$\mathbb C\to\mathcal R(L)\to\mathcal R(G)\to\mathcal R(F)\to\mathbb C$$

We saw above that as a right $\mathcal{R}(L)$-module, $\mathcal{R}(G)$ is projective and finitely generated. In the present setting however, Theorems 7.2.2 and 7.2.11 of \cite{mon} show that the module is in fact free of rank $|F|$. Thus, we have an isomorphism as follows:
$$End_{\mathcal R(L)}\mathcal R(G)\simeq M_{|F|}(\mathcal R(L))$$

Thus, the map $\pi$ in the statement is simply the one constructed in Proposition 2.6, and the stationarity property corresponds to the integration formula in Proposition 2.6. The lift from Peter-Weyl algebras to $C^*$-algebras follows as in Theorem 2.7.
\end{proof}

Observe that the above result covers three main situations, as follows:

\begin{proposition}
The construction in Theorem 3.1 provides us with a stationary model $\pi:C(G)\to M_F(C(L))$, with $F$ finite, and $L$ compact, in the following cases:
\begin{enumerate}
\item For any compact group $G$.

\item For any group dual $G=\widehat{\Gamma}$ which is such that $C(G)$ is of type I.

\item For any finite quantum group $G$.
\end{enumerate}
\end{proposition}

\begin{proof}
All these assertions are clear, as follows:

(1) Groups. Here the result holds indeed, because we can use the quotient map $G\to G$, which provides us with the model $\pi:C(G)\to M_1(C(G))$. 

(2) Group duals. Here, with notations from Thoma's theorem, we can use the quotient $\widehat{\Gamma}\to\widehat{\Lambda}$, which provides us with the model $\pi:C^*(\Gamma)\to M_F(C^*(\Lambda))$, where $F=\widehat{\Phi}$.

(3) Finite quantum groups. Here the result holds as well, because we can use the quotient $G\to\{1\}$, which provides us with the model $\pi:C(G)\to M_G(\mathbb C)$. 
\end{proof}

In general, Theorem 3.1 cannot be regarded as the ``quantum Thoma
theorem'', because it does not cover several examples of compact
quantum groups $G$ whose algebras $C(G)$ are known to be of type I. To
be more precise, we have the following result:

\begin{proposition}
We have stationary models $\pi:C(G)\to M_F(C(L))$, with $F$ finite, and $L$ compact, not necessarily coming from a cleft sequence, in the following cases:
\begin{enumerate}
\item For the non-classical subgroups $G\subset O_N^*$.

\item For the quantum permutation group $G=S_4^+$.

\item For the quotients $G\to H$ of the quantum groups $G$ having this property.
\end{enumerate}
\end{proposition}

\begin{proof}
Regarding the constructive part, all the results are well-known, as follows:

(1) It is known from \cite{bdu} that for such quantum groups we have stationary models $\pi:C(G)\to M_{\mathbb Z_2}(C(L))$, with $L\subset U_N$. We will discuss this in section 4 below.

(2) Here is is known from \cite{bco} that we have a model $\pi:C(S_4^+)\to M_{\mathbb Z_2\times\mathbb Z_2}(C(SO_3))$, which is stationary. Once again, we will discuss this, in section 6 below.

(3) Assuming indeed that we have an embedding $\rho:C(H)\subset C(G)$, we can compose $\rho$ with the stationary model for $C(G)$, and we obtain a stationary model for $C(H)$.

Regarding now the negative statements, in relation with the cleft sequences, the results here are once again well-known, and we will discuss them in sections 4-6 below.
\end{proof}

As a conclusion, Theorem 3.1 should be regarded as being just a first step towards a quantum Thoma theorem, waiting to be further generalized, as to cover for instance the examples in Proposition 3.3. We are still far away from some conjectural statement here, but in relation with all this, we can however formulate a conjecture, as follows:

\begin{conjecture}
For a compact quantum group $G$, the following are equivalent:
\begin{enumerate}
\item $C(G)$ is of type I.

\item $C(G)$ has a stationary model.
\end{enumerate}
\end{conjecture}

Regarding the evidence, the known examples of compact quantum groups as in (1) are those in Proposition 3.2 and Proposition 3.3. Thus, we have no counterexamples.

Finally, let us mention that this kind of conjecture does not make much sense if we remove the assumption $S^2=id$. Indeed, if the algebra $C(G)$ has a stationary model then the Haar integration must be a trace, and by \cite{wo3} we must have $S^2=id$.

Of course, understanding the structure of the compact quantum groups $G$, in the generalized sense of \cite{wo3}, with no $S^2=id$ assumption, whose associated algebras $C(G)$ are of type I is an interesting question. As a basic example here, the algebra $C(SU_q(N))$ with $q\in\mathbb R-\{0\}$, constructed in \cite{wo2}, is of type I, for any value of $q$, as shown in \cite{bra}. Our present techniques cannot be used for investigating such algebras, and this is why we formulated Definition 2.1 above as it is, with the assumption $S^2=id$ included.

We refer to \cite{bcv}, \cite{fim} for more on these topics.

\section{Cyclic models}

We restrict the attention in what follows to the matrix case. The formalism here, due once again to Woronowicz \cite{wo1}, \cite{wo2}, is particularly simple, as follows:

\begin{definition}
Assume that $A$ is a unital $C^*$-algebra, and $u\in M_N(A)$ is a unitary matrix, such that the following formulae define morphisms of $C^*$-algebras:
$$\Delta(u_{ij})=\sum_ku_{ik}\otimes u_{kj}\quad,\quad\varepsilon(u_{ij})=\delta_{ij}\quad,\quad S(u_{ij})=u_{ji}^*$$
We write then $A=C(G)=C^*(\Gamma)$, and call $G,\Gamma$ compact matrix quantum group (or compact quantum Lie group), respectively finitely generated discrete quantum group.
\end{definition}

Observe that the above morphisms $\Delta,\varepsilon,S$ satisfy the usual Hopf algebra axioms, along with the extra axiom $S^2=id$. As explained in \cite{wo3}, the simplifiability assumptions are satisfied as well, so this definition is compatible with Definition 2.1 above. This restricted formalism covers the compact Lie groups $G\subset U_N$, their $q$-deformations at $q=-1$, as well as the finitely generated discrete groups $\Gamma=<g_1,\ldots,g_N>$. See \cite{wo1}, \cite{wo2}.

We will be interested here in the quantum groups appearing via the half-liberation operation \cite{bb2}, \cite{bb3}, \cite{bi2}, \cite{bdu}. Let us first recall from \cite{bb3} that we have:

\begin{definition}
The half-liberated unitary group $U_N^*$ is the compact quantum group with standard coordinates $(u_{ij})_{i,j=1,\ldots,N}$ subject to the following conditions:
\begin{enumerate}
\item The matrices $u=(u_{ij})$ and $\bar{u}=(u_{ij}^*)$ are both unitaries.

\item The elements $\{ab^*,a^*b\}$, with $a,b\in\{u_{ij}\}$, all commute.
\end{enumerate}
\end{definition}

Here the relations (1) are those defining the free unitary group $U_N^+$, constructed by Wang in \cite{wa1}. As for the relations (2), the idea here is that associated to any closed subgroup $G\subset U_N^+$ are its left, right and full projective versions, having $p=u\otimes\bar{u}$, $q=\bar{u}\otimes u$, $r=p\oplus q$ as fundamental corepresentations. With this notion in hand, $U_N^*\subset U_N^+$ is the biggest closed subgroup having a classical full projective version. See \cite{bb3}.

The relation with the matrix models comes from the following fact:

\begin{proposition}
If $L$ is a compact group, having a $N$-dimensional unitary corepresentation $v$, and an order $K$ automorphism $\sigma:L\to L$, we have a matrix model
$$\pi:C(U_N^*)\to M_K(C(L))\quad,\quad u_{ij}\to\tau[v_{ij}^{(1)},\ldots,v_{ij}^{(K)}]$$
where $v^{(i)}(g)=v(\sigma^i(g))$, and where $\tau[x_1,\ldots,x_K]$ is obtained by filling the standard $K$-cycle $\tau\in M_K(0,1)$ with the elements $x_1,\ldots,x_K$. We call such models ``cyclic''.
\end{proposition}

\begin{proof}
The matrices $U_{ij}=\tau[v_{ij}^{(1)},\ldots,v_{ij}^{(K)}]$ in the statement appear by definition as follows, with the convention that all the blank spaces denote 0 entries:
$$U_{ij}=\begin{pmatrix}
&&&v_{ij}^{(1)}\\
v_{ij}^{(2)}\\
&\ddots\\
&&v_{ij}^{(K)}
\end{pmatrix}$$

The matrix $U=(U_{ij})$ is then unitary, and so is $\bar{U}=(U_{ij}^*)$. Thus, if we denote by $w=(w_{ij})$ the fundamental corepresentation of $C(U_N^+)$, we have a model as follows:
$$\rho:C(U_N^+)\to M_K(C(L))\quad,\quad w_{ij}\to U_{ij}$$

Now observe that the matrices $U_{ij}U_{kl}^*,U_{ij}^*U_{kl}$ are all diagonal, so in particular, they commute. Thus the above morphism $\rho$ factorizes through $C(U_N^*)$, as claimed.
\end{proof}

Following \cite{ban}, we say that a matrix model $\pi:C(G)\to M_K(C(X))$ is stationary on its image when its image coincides with its Hopf image. The terminology comes from the fact that, when this condition is satisfied, the stationarity property is automatic.

With this notion in hand, we can apply our Thoma type results, and we obtain:

\begin{theorem}
Any cyclic model $\pi:C(U_N^*)\to M_K(C(L))$ is stationary on its image, with the corresponding closed subgroup $[L]\subset U_N^*$, given by $Im(\pi)=C([L])$, being the quotient $L\rtimes\mathbb Z_K\to[L]$ having as coordinates the variables $u_{ij}=v_{ij}\otimes\tau$.
\end{theorem}

\begin{proof}
Assuming that $(L,\sigma)$ are as in Proposition 4.3, we have an action $\mathbb Z_K\curvearrowright L$, and we can therefore consider the following short exact sequence:
$$1\to\mathbb Z_K\to L\rtimes\mathbb Z_K\to L\to1$$

By performing the Thoma construction we obtain a model as follows, where $x^{(i)}=\tilde{\sigma}^i(x)$, with $\tilde{\sigma}:C(L)\to C(L)$ being the automorphism induced by $\sigma:L\to L$:
$$\rho:C(L\rtimes\mathbb Z_K)\subset M_K(C(L))\quad,\quad x\otimes\tau^i\to\tau^i[x^{(1)},\ldots,x^{(K)}]$$

Consider now the quotient quantum group $L\rtimes\mathbb Z_K\to[L]$ having as coordinates the variables $u_{ij}=v_{ij}\otimes\tau$. We have then a injective morphism, as follows:
$$\nu:C([L])\subset C(L\rtimes\mathbb Z_K)\quad,\quad u_{ij}\to v_{ij}\otimes\tau$$

By composing the above two embeddings, we obtain an embedding as follows:
$$\rho\nu:C([L])\subset M_K(C(L))\quad,\quad u_{ij}\to\tau[v_{ij}^{(1)},\ldots,v_{ij}^{(K)}]$$

Now since $\rho$ is stationary, and since $\nu$ commutes with the Haar funtionals as well, it follows that this morphism $\rho\nu$ is stationary, and this finishes the proof.
\end{proof}

We recall that $O_N^+\subset U_N^+$ is the closed subgroup obtained by assuming that the standard coordinates are self-adjoint. If we set $O_N^*=U_N^*\cap O_N^+$, the $O_N^*\subset O_N^+$ is the closed subgroup obtained by assuming that the standard coordinates satisfy the relations $abc=cba$. Moreover, it is known that we have an isomorphism $PO_N^*=PU_N$. See \cite{bb2}.

Now back to Theorem 4.4, when using $K=2$, and subgroups $L\subset U_N$ which are self-conjugate, we recover the following result, from \cite{bdu}:

\begin{proposition}
For any non-classical subgroup $G\subset O_N^*$ we have a stationary model
$$\pi:C(G)\to M_2(C(L))\quad,\quad u_{ij}=\begin{pmatrix}0&v_{ij}\\ \bar{v}_{ij}&0\end{pmatrix}$$
where $L\subset U_N$, with coordinates denoted $v_{ij}$, is the lift of $PG\subset PO_N^*=PU_N$.
\end{proposition}

\begin{proof}
Assume first that $L\subset U_N$ is self-conjugate, in the sense that $g\in L\implies\bar{g}\in L$. If we consider the order 2 automorphism of $C(L)$ induced by $g_{ij}\to\bar{g}_{ij}$, we can apply Theorem 4.4, and we obtain a stationary model, as follows:
$$\pi:C([L])\subset M_2(C(L))\quad,\quad u_{ij}\otimes1=\begin{pmatrix}0&v_{ij}\\ \bar{v}_{ij}&0\end{pmatrix}$$

The point now is that, as explained in \cite{bdu}, any non-classical subgroup $G\subset O_N^*$ must appear as $G=[L]$, for a certain self-conjugate subgroup $L\subset U_N$. Moreover, since we have $PG=P[L]$, it follows that $L\subset U_N$ is the lift of $PG\subset PO_N^*=PU_N$, as claimed. 
\end{proof}

In the unitary case now, and with $K\in\mathbb N$ being arbitrary, we recall from \cite{bb2} that $U_N^*$ has a certain ``arithmetic version'' $U_{N,K}^*\subset U_N^*$, obtained by imposing some natural length $2K$ relations on the standard coordinates. As basic examples here, at $K=1$ we have $U_{N,1}^*=U_N$, the defining relations being $ab=ba$ with $a,b\in\{u_{ij},u_{ij}^*\}$, and at $K=2$ we have $U_{N,2}^*=U_N^{**}$, with the latter quantum group being the one from \cite{bdu}, appearing via the relations $ab\cdot cd=cd\cdot ab$, for any $a,b,c,d\in\{u_{ij},u_{ij}^*\}$. See \cite{bb2}, \cite{bb3}.

We have the following result, which clarifies the relation with \cite{bb2}:

\begin{proposition}
For any subgroup $G\subset U_{N,K}^*$ which is $K$-symmetric, in the sense that $u_{ij}\to e^{2\pi i/K}u_{ij}$ defines an automorphism of $C(G)$, we have a stationary model
$$\pi:C(G)\to M_K(C(L))\quad,\quad u_{ij}\to\tau[v_{ij}^{(1)},\ldots,v_{ij}^{(K)}]$$
with $L\subset U_N^K$ being a closed subgroup which is symmetric, in the sense that it is stable under the cyclic action $\mathbb Z_K\curvearrowright U_N^K$.
\end{proposition}

\begin{proof}
Assuming that $L\subset U_N^K$ is symmetric in the above sense, we have representations $v^{(i)}:L\subset U_N^K\to U_N^{(i)}$ for any $i$, and the cyclic action $\mathbb Z_K\curvearrowright U_N^K$ restricts into an order $K$ automorphism $\sigma:L\to L$. Thus we can apply Theorem 4.4, and we obtain a certain closed subgroup $[L]\subset U_{N,K}^*$, having a stationary model as in the statement.

Conversely now, assuming that $G\subset U_{N,K}^*$ is $K$-symmetric, the main result in \cite{bb2} applies, and shows that we must have $C(G)\subset C(L)\rtimes\mathbb Z_K$, for a certain closed subgroup $L\subset U_N^K$ which is symmetric. But this shows that we have $G=[L]$, and we are done.
\end{proof}

In general, Proposition 4.6 above does not close the discussion. One interesting modelling question is for instance that concerning $U_N^*$ itself. Indeed, this quantum group is conjectured to be coamenable, cf. \cite{bb3}, and finding any kind of ``generalized matrix model'' for it would probably prove this conjecture, which looks non-trivial.

\section{Quantum permutations}

We discuss in what follows the quantum permutation group case, in connection with some previous work from \cite{ban}, \cite{bfr}, \cite{bne}. We recall that a magic unitary matrix is a square matrix over a $C^*$-algebra, $u\in M_N(A)$,  whose entries are projections, summing up to $1$ on each row and each column. The following key definition is due to Wang \cite{wa2}:

\begin{definition}
$C(S_N^+)$ is the universal $C^*$-algebra generated by the entries of a $N\times N$ magic unitary matrix $u=(u_{ij})$, with the morphisms given by
$$\Delta(u_{ij})=\sum_ku_{ik}\otimes u_{kj}\quad,\quad\varepsilon(u_{ij})=\delta_{ij}\quad,\quad S(u_{ij})=u_{ji}$$
as comultiplication, counit and antipode. 
\end{definition}

This algebra satisfies Woronowicz's axioms, so $S_N^+$ is a compact quantum group, called quantum permutation group. The terminology comes from the fact that we have an inclusion $S_N\subset S_N^+$, coming from the representation $\pi:C(S_N^+)\to C(S_N)$ given by:
$$u_{ij}\to\chi\left(\sigma\in S_N\Big|\sigma(j)=i\right)$$

This inclusion is known to be an isomorphism at $N=2,3$, but not at $N\geq4$, where $S_N^+$ is non-classical, and infinite. Moreover, it is known that we have $S_4^+\simeq SO_3^{-1}$, and that any $S_N^+$ with $N\geq4$ has the same fusion semiring as $SO_3$. See \cite{bi1}, \cite{wa2}.

Any closed subgroup $G\subset S_N^+$ can be thought of as ``acting'' on the set $\{1,\ldots,N\}$, and one can talk about the orbits of this action. The theory here was developed in \cite{bi1}, and also recently in \cite{bfr}. In what follows, we will need the following notion:

\begin{definition}
Given a closed subgroup $G\subset S_N^+$, with fundamental corepresentation $u=(u_{ij})$, consider the equivalence relation on $\{1,\ldots,N\}$ given by $i\sim j$ when $u_{ij}\neq0$.
\begin{enumerate}
\item The equivalence classes for this relation are called orbits of $G$.

\item We call $G$ quasi-transitive when these orbits have the same size.
\end{enumerate} 
\end{definition}

Here the fact that $\sim$ is indeed an equivalence relation comes by applying $\Delta,\varepsilon,S$ to a formula of type $u_{ij}\neq0$. Observe also that in the classical case, $G\subset S_N$, we have $u_{ij}=\chi(\sigma\in G|\sigma(j)=i)$, so we obtain indeed the orbits of the action $G\curvearrowright\{1,\ldots,N\}$. Finally, in the case where we have just one orbit, which amounts in saying that $u_{ij}\neq0$ for any $i,j$, we say that $G$ is transitive. For details here, see \cite{bfr}.

In connection now with matrix models, we have the following construction, which goes back to \cite{ban}, \cite{bne} in the transitive case, and to \cite{bfr} in the general quasi-transitive case:

\begin{definition}
Assume that $G\subset S_N^+$ is quasi-transitive, with orbits of size $K$.
\begin{enumerate}
\item A matrix model $\pi:C(G)\to M_K(C(X))$, mapping $u_{ij}\to P_{ij}^x$, is called quasi-flat when $u_{ij}\neq0$ implies $rank(P_{ij}^x)=1$ for any $x\in X$.
\item The universal quasi-flat model for $C(G)$, obtained via the Tannakian relations defining $C(G)$, is denoted $\pi_G:C(G)\to M_K(C(X_G))$.
\end{enumerate}
\end{definition}

To be more precise, the existence and uniqueness of the universal quasi-flat model is clear for $G=S_N^+$ itself, the model space here being the submanifold $X_N\subset M_N(P^{N-1}_\mathbb C)$, where $P^{N-1}_\mathbb C$ is identified with the space of rank 1 projections in $M_N(\mathbb C)$, defined by the equations stating that the vectors on the rows and columns must be pairwise orthogonal. In the general transitive case $G\subset S_N^+$ we must further impose the Tannakian conditions $T\in Hom(u^{\otimes k},u^{\otimes l})$ which define the quotient algebra $C(S_N^+)\to C(G)$, and this leads to a certain smaller algebraic manifold $X_G\subset X_N$. Finally, in the quasi-transitive case, where $G\subset S_N^+$ has orbits of size $K$, with $N=KM$, a similar construction applies, and we are led to a model space of type $X_G\subset X_K^M\subset X_N$. See \cite{ban}, \cite{bfr}, \cite{bne}.

In relation with stationarity, we first have the following result, coming from \cite{bfr}:

\begin{proposition}
For a quasi-transitive group $G\subset S_N$, with orbits having size $K$, the following are equivalent:
\begin{enumerate}
\item The universal quasi-flat model $\pi:C(G)\to M_K(C(X_G))$ is stationary.

\item The universal quasi-flat model space is non-empty, $X_G\neq\emptyset$.

\item $\exists\,\sigma_1,\ldots,\sigma_K\in G$ such that $\sigma_1(i),\ldots,\sigma_K(i)$ are distinct, $\forall i\in\{1,\ldots,N\}$.
\end{enumerate}
In addition, these conditions are not automatically satisfied, and fail for instance for a certain copy $\mathbb Z_2\times\mathbb Z_2\subset S_6$.
\end{proposition}

\begin{proof}
The idea here is that the quasi-flat models $\pi:C(G)\to M_K(\mathbb C)$ can be parametrized by pairs $(P,L)$, where $P=(P_1,\ldots,P_K)$ is a partition of the unity of $M_K(\mathbb C)$ with rank 1 projections, and where $L\in M_N(*,1,\ldots,K)$ is a sparse Latin square. With this observation in hand, both $(1)\iff(2)$ and $(2)\iff(3)$ follow. See \cite{bfr}. 

Regarding the last assertion, consider the transpositions $\alpha=(12),\beta=(34),\gamma=(56)$, inside the symmetric group $S_6$. The group $G=\{1,\alpha\beta,\alpha\gamma,\beta\gamma\}$, which is isomorphic to $\mathbb Z_2\times\mathbb Z_2$, is then quasi-transitive, with orbits having size 2. On the other hand, since $G-\{1\}$ contains no derangement, we cannot find elements $\sigma_1,\sigma_2\in G$ such that $\sigma_1(i)\neq\sigma_2(i)$ for any $i$, because this would tell us that $\sigma_2^{-1}\sigma_1\in G-\{1\}$ is a derangement. Thus, the above condition (3) fails for this group $G$, and this finishes the proof.
\end{proof}

In general, the study of the above conditions is something non-trivial. We intend to come back to these questions, in the transitive case, in a forthcoming paper.

In order to investigate now the general case, we will need:

\begin{proposition}
For $G\subset S_N^+$, the algebra $Fix(u)=\{\xi\in\mathbb C^N|u\xi=\xi\}$ consists of the vectors $\xi\in\mathbb C^N\simeq C(1,\ldots,N)$ which are constant on the orbits of $G$. In particular: 
\begin{enumerate}
\item Assuming that $G$ is transitive, in the sense that $u_{ij}\neq0$ for any $i,j$, we have $Fix(u)=\mathbb C\eta$, where $\eta\in\mathbb C^N$ is the all-one vector.

\item Assuming that $G$ is quasi-transitive, we have $Fix(u)=span(\eta_1,\ldots,\eta_M)$, where $\eta_i\in\mathbb C^N=(\mathbb C^K)^{\oplus M}$ has $1$ entries on the $i$-th summand, and $0$ entries elsewhere.
\end{enumerate}
\end{proposition}

\begin{proof}
Consider the standard coaction $\alpha:\mathbb C^N\to C(G)\otimes\mathbb C^N$, given by $\alpha(\delta_i)=\sum_ju_{ij}\otimes\delta_j$. The algebra $Fix(u)$ is then the fixed point algebra of this coaction, namely:
$$(\mathbb C^N)^\alpha=\left\{\xi\in\mathbb C^N\big|\alpha(\xi)=1\otimes\xi\right\}$$

On the other hand, the general results in \cite{bi1} show that, via the identification $\mathbb C^N=C(1,\ldots,N)$ from the statement, this latter algebra is given by:
$$(\mathbb C^N)^\alpha=\left\{\xi\in C(1,\ldots,N)\big|i\sim j\implies \xi(i)=\xi(j)\right\}$$

Thus, we obtain the result, and then its particular cases (1,2), as stated.
\end{proof}

With the above result in hand, we can now prove:

\begin{theorem}
Assuming that $G\subset S_N^+$ is quasi-transitive, with orbits having size $K$, any stationary model $\pi:C(G)\to M_K(C(X))$ with $X$ connected is automatically quasi-flat.
\end{theorem}

\begin{proof}
If we denote the matrix model map by $u_{ij}\to P_{ij}^x$, the stationarity assumption, applied on the standard coordinates, shows that we have:
$$\int_Gu_{ij}=\int_Xtr(P_{ij}^x)dx$$

We use now the well-known fact, coming from \cite{wo1}, that the matrix $Q=(\int_Gu_{ij})_{ij}$ formed by the elements on the left is the orthogonal projection onto $Fix(u)$. By combining this observation with the results in Proposition 5.5, we succesively conclude that:

-- In the transitive case we have $Q=\frac{1}{N}J_N$, where $J_N$ is the all-one matrix. 

-- In the general case, we have $Q=(\frac{1}{K}J_K)^{\oplus M}$, where $M=N/K$. 

With these formulae in hand, by getting now back to our equality coming from the stationarity condition, this simply becomes:
$$u_{ij}\neq0\implies\int_Xtr(P_{ij}^x)dx=\frac{1}{K}$$

Now since the functions $x\to tr(P_{ij}^x)$ are locally constant, since $X$ was assumed to be connected, this condition tells us that the model is quasi-flat, as claimed.
\end{proof}

These results suggest the following refinement of Conjecture 3.4:

\begin{conjecture}
For a closed subgroup $G\subset S_N^+$, with orbits having size $K$, and satisfying a supplementary transitivity type assumption, the following are equivalent:
\begin{enumerate}
\item $C(G)$ is of type I.

\item $C(G)$ has a stationary model.

\item $C(G)$ has a stationary $K\times K$ model.

\item The universal quasi-flat model for $C(G)$ is stationary.
\end{enumerate}
\end{conjecture}

In this statement $(4)\implies(3)\implies(2)\implies(1)$ are trivial, $(1)\implies(2)$ is expected to hold without assumptions, as stated in Conjecture 3.4, and $(3)\implies(4)$ does hold, under a mild assumption, as shown by Theorem 5.6. Thus, the conjecture is that, under suitable assumptions, we should have $(2)\implies(3)$. In the classical case the needed assumptions are those in Proposition 5.4. In the general case, however, we do not know what the correct assumptions are, and this statement is the best one that we have.

\section{Uniform groups}

In this section we discuss the group dual case, $\widehat{\Gamma}\subset S_N^+$, with $\Gamma$ being classical. These group duals were classified by Bichon in \cite{bi1}, the result being as follows:

\begin{proposition}
The group duals $\widehat{\Gamma}\subset S_N^+$ appear as follows:
\begin{enumerate}
\item Given integers $K_1,\ldots,K_M$ satisfying $K_1+\ldots+K_M=N$, the dual of any quotient group $\mathbb Z_{K_1}*\ldots*\mathbb Z_{K_M}\to\Gamma$ appears as a closed subgroup $\widehat{\Gamma}\subset S_N^+$. 

\item By refining if necessary the partition $N=K_1+\ldots+K_M$, we can always assume that the $M$ morphisms $\mathbb Z_{K_i}\to\Gamma$ are all injective.

\item Assuming that the partition $N=K_1+\ldots+K_M$ is refined, as above, this partition is precisely the one describing the orbit structure of $\widehat{\Gamma}\subset S_N^+$.

\item Modulo a conjugation by a permutation matrix $W\in S_N$, we obtain in this way all the group dual subgroups $\widehat{\Gamma}\subset S_N^+$.
\end{enumerate}
\end{proposition}

\begin{proof}
The idea for (1) is that we have embeddings $\widehat{\mathbb Z}_{K_i}\simeq\mathbb Z_{K_i}\subset S_{K_i}\subset S_{K_i}^+$, and by performing a free product construction, we obtain an embedding as follows:
$$\widehat{\Gamma}\subset\widehat{\mathbb Z_{K_1}*\ldots*\mathbb Z_{K_M}}\subset S_N^+$$

To be more precise, the magic unitary that we get is as follows, where $F_i=\frac{1}{\sqrt{K_i}}(w_i^{ab})_{ab}$ with $w_i=e^{2\pi i/K_i}$, and  $V_i=(g_i^a)_a$, with $g_i$ being the standard generator of $\mathbb Z_{K_i}$:
$$u=diag(u_i)\quad,\quad u_i=\frac{1}{\sqrt{K_i}}\begin{pmatrix}(F_iV_i)_0&\ldots&(F_iV_i)_{K_i-1}\\ (F_iV_i)_{K_i-1}&\ldots&(F_iV_i)_{K_i-2}\\ \ldots&\ldots&\ldots\\ (F_iV_i)_1&\ldots&(F_iV_i)_0\end{pmatrix}$$

Regarding (2,3,4), the idea here is that the orbit structure of any $\widehat{\Gamma}\subset S_N^+$ produces a partition $N=K_1+\ldots+K_M$, and then a quotient map $\mathbb Z_{K_1}*\ldots*\mathbb Z_{K_M}\to\Gamma$. See \cite{bi1}.
\end{proof}

Regarding now the quasi-transitive case, and our modelling questions, we have:

\begin{proposition}
The quasi-transitive group duals $\widehat{\Gamma}\subset S_N^+$, with orbits having $K$ elements, have the following properties:
\begin{enumerate}
\item These come from the quotients $\mathbb Z_K^{*M}\to\Gamma$, having the property that the corresponding $M$ morphisms $\mathbb Z_K^{(i)}\subset\mathbb Z_K^{*M}\to\Gamma$ are all injective.

\item For such a quotient, a matrix model $\pi:C^*(\Gamma)\to M_K(\mathbb C)$ is quasi-flat if and only if it is stationary on each subalgebra $C^*(\mathbb Z_K^{(i)})\subset C^*(\Gamma)$.
\end{enumerate}
\end{proposition}

\begin{proof}
The first assertion follows from Proposition 6.1. Regarding the second assertion, consider an arbitrary matrix model $\pi:C^*(\Gamma)\to M_K(\mathbb C)$, mapping $g_i\to U_i$, where $g_i$ is the standard generator of $\mathbb Z_K^{(i)}$. With notations from the proof of Proposition 6.1, the images of the nonzero standard coordinates on $\widehat{\Gamma}\subset S_N^+$ are mapped as follows:
$$\pi:\frac{1}{\sqrt{K}}(FV_i)_c\to\frac{1}{\sqrt{K}}(FW_i)_c$$

Here $V_i=(g_i^a)_a$, $W_i=(U_i^a)_a$, and $F=\frac{1}{\sqrt{K}}(w^{ab})_{ab}$ with $w=e^{2\pi i/K}$. With this formula in hand, the flatness condition on $\pi$ simply states that we must have:
$$Tr((FW_i)_c)=\sqrt{K}\quad,\quad\forall i,\forall c$$

In terms of the trace vectors $T_i=(Tr(U_i^a))_a$ this condition becomes $FT_i=\sqrt{K}\xi$, where $\xi\in\mathbb C^K$ is the all-one vector. Thus we must have $T_i=\sqrt{K}F^*\xi$, which reads:
$$\begin{pmatrix}Tr(1)\\ Tr(U_i)\\\ldots\\ Tr(U_i^{K-1})\end{pmatrix}=\sqrt{K}F^*\begin{pmatrix}1\\ 1\\\ldots\\1\end{pmatrix}=\begin{pmatrix}K\\ 0\\\ldots\\0\end{pmatrix}\quad,\quad\forall i$$
 
In other words, we have reached to the conclusion that $\pi$ is flat precisely when its restrictions to each subalgebra $C^*(\mathbb Z_K^{(i)})\subset C^*(\Gamma)$ are stationary, as claimed.
\end{proof}

We would like to end our study with a purely group-theoretical formulation of these results, and of some related questions, that we believe of interest. Let us start with:

\begin{definition}
A discrete group $\Gamma$ is called uniform when:
\begin{enumerate}
\item $\Gamma$ is finitely generated, $\Gamma=<g_1,\ldots,g_M>$.

\item The generators $g_1,\ldots,g_M$ have common order $K<\infty$.

\item $\Gamma$ appears as an intermediate quotient $\mathbb Z_K^{*M}\to\Gamma\to\mathbb Z_K^M$.

\item We have an action $S_M\curvearrowright\Gamma$, given by $\sigma(g_i)=g_{\sigma(i)}$.
\end{enumerate}
\end{definition}

Here the conditions (1-3) basically come from Bichon's work \cite{bi1}, via Proposition 6.2 (1) above, and together with some extra considerations from \cite{bfr}, which prevent us from using groups of type $\Gamma=(\mathbb Z_K*\mathbb Z_K)\times\mathbb Z_K$, we are led to the condition (4) as well.

Observe that some of the above conditions are technically redundant, with (4) implying that the generators $g_1,\ldots,g_M$ have common order, as stated in (2), and also with (3) implying that the group is finitely generated, with generators having finite order.

We have as well the following notion, which is once again group-theoretical:

\begin{definition}
If $\Gamma$ is uniform, as above, a unitary representation $\rho:\Gamma\to U_K$ is called quasi-flat when the eigenvalues of each $U_i=\rho(g_i)\in U_K$ are uniformly distributed.
\end{definition}

To be more precise, assuming that $\Gamma=<g_1,\ldots,g_M>$ with $ord(g_i)=K$ is as in Definition 6.3, any unitary representation $\rho:\Gamma\to U_K$ is uniquely determined by the images $U_i=\rho(g_i)\in U_K$ of the standard generators. Now since each of these unitaries satisfies $U_i^K=1$, its eigenvalues must be among the $K$-th roots of unity, and our quasi-flatness assumption states that each eigenvalue must appear with multiplicity $1$.

With these notions in hand, we have the following result:

\begin{theorem}
If $\Gamma=<g_1,\ldots,g_M>$ is uniform, with $ord(g_i)=K$, a matrix model $$\pi:C^*(\Gamma)\to M_K(C(X))$$
is quasi-flat in the sense of Definition 5.3 precisely when the associated unitary representation $\rho:\Gamma\to C(X,U_K)$ has quasi-flat fibers, in the sense of Definition 6.4.
\end{theorem} 

\begin{proof}
According to Proposition 6.2 (2) above, the model is quasi-flat precisely when the compositions $\pi_i:C^*(\mathbb Z_K^{(i)})\subset C^*(\Gamma)\to M_K(C(X))$ are all stationary. 

On the other hand, as already observed in the proof of Proposition 6.2, a matrix model $\rho:C^*(\mathbb Z_K)\to M_K(C(X))$ is stationary precisely when the unitary $U=\rho(g)$, where $g$ is the standard generator of $\mathbb Z_K$, satisfies the following condition:
$$\begin{pmatrix}tr(1)\\ tr(U)\\\ldots\\ tr(U^{K-1})\end{pmatrix}=\begin{pmatrix}1\\ 0\\\ldots\\0\end{pmatrix}$$

Thus, such a model is stationary precisely when the eigenvalues of $U$ are uniformly distributed, over the $K$-th roots of unity. We conclude that $\pi$ is quasi-flat precisely when the eigenvalues of each $U_i=\rho(g_i)$ are uniformly distributed, as in Definition 6.4.
\end{proof}

Finally, we have the following conjecture, which would refine Conjecture 5.7:

\begin{conjecture}
Assuming that $\mathbb Z_K^{*M}\to\Gamma\to\mathbb Z_K^M$ is uniform, in some strong sense:
\begin{enumerate}
\item The model space $X_G$ for the group dual $G=\widehat{\Gamma}\subset S_N^+$ is an homogeneous space.

\item In the virtually abelian case, the Haar measure on $X_G$ produces the stationarity.
\end{enumerate}
\end{conjecture}

Here both statements are non-trivial. Some verifications of (1) were performed in \cite{bfr}, for certain basic classes of uniform groups, including the extremal cases $\Gamma=\mathbb Z_K^{*M}$ and $\Gamma=\mathbb Z_K^M$, and some amalgamation-theoretic variations of these examples. In general, however, all this looks quite non-trivial, and might actually need some stronger uniformity assumptions, as those used in \cite{rwe}. As for (2), we have no results here yet.

We believe that a good framework for such questions is the unitary easy quantum group setting \cite{twe}. Indeed, as explained in \cite{bb3}, associated to any intermediate easy quantum group $O_N\subset G\subset U_N^+$ is a certain ``noncommutative geometry'', which includes in particular a certain noncommutative torus $T=\widehat{\Gamma}$. Now when restricting the attention to the geometries which are ``hybrid'' in the sense of \cite{bb3}, in the sense that they are neither real, nor complex, the dual of the torus $\Gamma=\widehat{T}$ is uniform in the sense of Definition 6.3 above, and our feeling is that for this class of discrete groups, the conclusions of Conjecture 6.6 should hold as stated. We intend to discuss all this in a future paper.

Finally, an interesting problem, which would probably provide some good input for our various conjectures, is that of explicitly computing the universal quasi-flat models for the closed subgroups $G\subset S_4^+$. These subgroups, which are all coamenable, were fully classified in \cite{bb1}, and most of them can be investigated by using the above results. The examples which are not covered yet by our results consist in certain finite quantum groups, appearing as cocycle twists \cite{eva}, plus $O_2^{-1},SO_3^{-1}$, which can be probably investigated by using the fibers of the Pauli matrix representation \cite{bb1}, \cite{bco}, \cite{end}, \cite{joz}.

\end{document}